\documentclass[12pt]{article}
\newcommand{\version}{\today}

\usepackage{amsthm,amsfonts,amsmath, amscd}
\usepackage{color}

\usepackage{amsmath}
\usepackage[english]{babel}
\usepackage{latexsym}
\usepackage{amssymb}
\usepackage{amscd}
\usepackage{amsgen,amstext,amsbsy,amsopn}
\usepackage{math rsfs}
\usepackage{amsthm,epsfig,graphicx,graphics}
\usepackage[latin1]{inputenc}
\usepackage{xspace}
\usepackage{amsxtra}

\usepackage{color}

\usepackage{amstext}
\usepackage{amsxtra}
\usepackage{amsmath}
\usepackage{amssymb}
\usepackage{bbold}
\usepackage{fancybox}
\usepackage{amsthm}
\usepackage{amsfonts}
\usepackage{amstext}
\usepackage{graphicx}
\usepackage{color}
\usepackage{pst-all}
\usepackage{mathrsfs}
\usepackage{dsfont}
\usepackage{latexsym}

\usepackage{amscd}

\usepackage{amsgen,amstext,amsbsy,amsopn}
\usepackage{math rsfs}
                            
\swapnumbers
                              
\theoremstyle{plain}
\newtheorem{thm}{THEOREM}[section]
\newtheorem{lm}[thm]{LEMMA}

\theoremstyle{definition}

\newtheorem{remark}[thm]{Remark}
\theoremstyle{remark}
\newcommand{\upchi}{\raise1pt\hbox{$\chi$}}
\newcommand{\R}{{\mathord{\mathbb R}}}

\newcommand{\hn}{{\mathord{\widehat{n}}}}

\newcommand{\F}{{\mathcal{F}}}
\newcommand{\E}{{\mathcal{E}}}

\renewcommand{\|}{{\Vert}}
\numberwithin{equation}{section}
\newcommand{\un}{{\rm 1\kern -2.5pt l}}

    \newcommand{\one}{{\bf 1}}

\begin{document}
\markboth{\scriptsize{CL \version}}{\scriptsize{CL December 13, 2017}}
\def\mn{{\bf M}_n}
\def\hn{{\bf H}_n}
\def\hnp{{\bf H}_n^+}
\def\hmnp{{\bf H}_{mn}^+}
\def\H{{\mathcal H}}
\title{{\sc A dual form of the sharp Nash inequality and its weighted 
generalization}}
\author{
\vspace{5pt}  Eric A. Carlen$^{1}$  and Elliott H. Lieb$^{2}$ \\
\vspace{5pt}\small{$^{1}$ Department of Mathematics, Hill Center,}\\[-6pt]
\small{Rutgers University,
110 Frelinghuysen Road
Piscataway NJ 08854-8019 USA}\\
\vspace{5pt}\small{$^{2}$ Departments of Mathematics and Physics, Jadwin
Hall, Princeton University}\\[-6pt]
\small{ Washington Rd.,  Princeton, NJ  08544.}\\[-6pt]
 }

\date{December 13, 2017}
\maketitle

\footnotetext [1]{Work partially supported by U.S.
National Science Foundation grant  DMS 1501007.}

\footnotetext [2]{Work partially supported by U.S. National
Science Foundation
grant PHY 1265118 \hfill\break
\copyright\,  2017 by the authors. This paper may be reproduced, in its
entirety, for non-commercial purposes.
}

\begin{abstract} The well known  duality 
between the Sobolev inequality and the Hardy-Littlewood-Sobolev 
inequality suggests that the 
Nash inequality could also have an interesting dual form, even though the 
Nash inequality relates three norms instead of two. 
We provide such a dual form here with sharp constants. 
This 
dual inequality relates the $L^2$ norm 
to  the infimal convolution of the $L^\infty $ and  $H^{-1}$ norms.  
The computation of this infimal convolution is a minimization problem, 
which we
solve explicitly, thus providing a new proof of the sharp Nash inequality 
itself. This proof, via duality,  also yields the sharp form of some new,
weighted generalizations of the Nash inequality as well as  the dual of 
these 
weighted variants. 
\end{abstract}

\medskip
\centerline{Key Words: duality, Nash inequality, infimal convolution.}

\bigskip


\maketitle

\section{Introduction}  

Our focus is on a dual form of the Nash inequality for functions on $\R^n$
that  bounds
the infimal convolution of the $L^\infty$ norm  and the squared
$H^{-1}$ norm in terms of the $L^2$ norm as follows:
\begin{equation}\label{nash8A}
L_n\|g\|_{2}^{\frac{2n+4}{n+4}}  \geq \inf_{ h\in L^{2n/(n+2}(\R^n)}\left\{  
\frac12 \| (-\Delta)^{-1/2}(g-h)\|^2_{2} +  \|h\|_{\infty}  
\ \right\} \ . 
\end{equation}
The sharp constant $L_n$ is  specified below in Theorem \ref{main1}.  (See 
the remarks at the beginning of Section~\ref{sec2} on the meaning of 
$\|(-\Delta)^{1/2}f\|_2$  in dimensions $1$ and $2$.)

The Nash inequality is one member of the  family 
of inequalities known as 
the {\em Gagliardo-Nirenberg-Sobolev} (GNS) {\em inequalities} \cite{G,Ni}: For 
all $1\leq p < q \leq 2n/(n-2)$ ($1 \leq p <  q <\infty$ for $n=1,2$),  there is a finite constant $C$ such that for all locally integrable functions $u$ on $\R^n$
with a distributional gradient that is a square integrable function,
\begin{equation}\label{gnsgen}
\|u\|_q \leq  C\|u\|_p^{1-\theta}\|\nabla u\|_2^\theta\ 
\end{equation}
where
\begin{equation}\label{gnsgen2}
\frac{1}{q} = \frac{\theta(n-2)}{2n} + \frac{1-\theta}{p}\ ,
\end{equation}
and $\|\cdot\|_p$ denotes the $L^p$ norm with respect to Lebesgue measure on $\R^n$. (A more general version, not discussed here, 
allows for the consideration of $\|\nabla u\|_r$ for $r\neq 2$.)
For $n\geq 3$ and $q = 2n/(n-2)$, in which case  (\ref{gnsgen2}) gives $\theta =1$, (\ref{gnsgen}) reduces to the {\em Sobolev inequality}
\begin{equation}\label{sharpsob}
\|u\|_{2n/(n-2)}^2 \leq S_n \|\nabla u\|_2^2\ ,
\end{equation}
for which the sharp constant $S_n$ was determined by Aubin and Talenti \cite{Au,Ta}.  There are only a few other cases of (\ref{gnsgen}) in which the sharp constant is known: 
One is that in which 
$p= r+1$ and $q= 2r$ for some $r>0$,  This family of sharp inequalities is due 
to Del Pino and Dolbeault \cite{DD}
The other family of sharp constants is for the Nash inequality \cite{Na}, in which $p=1$ and $q=2$, and hence $\theta = n/(n+2)$.
The sharp constants in this case were found by Carlen and Loss \cite{CL}.
It is well known that the sharp Sobolev inequality (\ref{sharpsob}) is equivalent, by duality, 
 to the sharp 
{\em Hardy-Littlewood-Sobolev} (HLS) {\em inequality} \cite[Thm 8.3]{LL},
and that this equivalence extends  to a more general form of the Sobolev inequality: For all $n\geq 2$ and all $0 < \alpha < n/2$, there is a constant $C_{n,\alpha}$ so that 
\begin{equation}\label{sobgen}
  \|u\|_{2n/(n-2\alpha)}^2  \leq C_{n,\alpha} 
\|(-\Delta)^{\alpha/2} u\|_2^2 \, 
\end{equation}
which reduces to (\ref{gnsgen}) for $\theta = 1$ when $\alpha =1$.   For all (real)
$v\in L^{2n/(n+2\alpha)}$, and all (real)  $u\in L^{2n/(n-2\alpha)}$,  (\ref{sharpsob}) is equivalent to 
$$\int_{\R^n} vu{\rm d}x -  \frac12 \|u\|_{2n/(n-2\alpha)}^2   \geq   \int_{\R^n} vu{\rm d}x  - \frac12 C_{n,\alpha} \|(-\Delta)^{\alpha/2} u\|_2^2\ .$$
Taking the supremum over $u$, which amounts to computing the Legendre transforms (see Rockafellar \cite{R}) of the  
proper lower semicontinuous convex functionals on both sides of 
(\ref{sharpsob}), one obtains
the general form of the HLS inequality:
\begin{equation}\label{hlsgen}
  \|v\|_{2n/(n+2\alpha)}^2  \geq C_{n,\alpha}^{-1}
\|(-\Delta)^{-\alpha/2} v\|_2^2 \ ,
\end{equation}
where  $ C_{n,\alpha}$, is the same consstant as in the Sobolev inequality 
 \eqref{sobgen}. See \cite{L83,LL} for the explicit value.

Since the Legendre transform is involutive for proper 
lower semicontinuous convex functionals, the argument is reversible, so that (\ref{sobgen}) 
and (\ref{hlsgen}) are equivalent,
and there is even a one-to-one correspondence between cases of equality in these two inequalities.

This duality is  useful: The 
sharp form of  (\ref{hlsgen}) was first obtained by Lieb \cite{L83}, and 
then, by this duality, the sharp form of (\ref{sobgen}) 
follows immediately. It is  natural that the sharp form of (\ref{sobgen}) 
was first obtained in this
indirect manner because 
 the functional $u\mapsto \|(-\Delta)^{\alpha/2}u\|_2$ is {\em not} monotone decreasing under 
spherically symmetric decreasing rearrangement for $\alpha > 1$, but  monotonicity under rearrangement 
 is valid in   HLS inequality
(\ref{hlsgen}) for all $0 < \alpha < n/2$, Thus, the dual relation between (\ref{sobgen}) and (\ref{hlsgen}) means that extremals of  
(\ref{sobgen}) must be translates of symmetric decreasing functions despite the non-monotonicity of  $u\mapsto \|(-\Delta)^{\alpha/2}u\|_2$
under rearrangement.

It is natural, therefore, to ask what are the dual forms of the 
various GNS 
inequalities with $\theta <1$ in \eqref{gnsgen}, recalling that $\theta=1$ 
corresponds to the usual Sobolev inequality. 
The first step is to produce an equivalent form of (\ref{gnsgen}) in which 
the right side is a convex function of $u$; even at this point there 
are choices to be made.  One way is to use the 
arithmetic-geometric mean inequality to deduce
\begin{equation}\label{gnsgen4}
\|u\|_q^2 \leq  (1-\theta) C^{2/(1-\theta)}\|u\|_p^2 + \theta \|\nabla u\|_2^2\ .
\end{equation}
since the two terms on the right side of (\ref{gnsgen4}) scale as different powers of $\lambda$ when $u(x)$ is replaced by
$u(\lambda x)$, one readily recovers (\ref{gnsgen}) from  (\ref{gnsgen4}). 

A second way is to introduce the functional $\Phi_p(u)$ defined by
\begin{equation}\label{phidef}
\Phi_p(u) = \begin{cases} 0 & \|u\|_p \leq 1\\ \infty & \|u\|_p > 1\ 
.\end{cases}\ 
\end{equation}
Then (\ref{gnsgen}) is also equivalent to the inequality
\begin{equation}\label{gnsgen5}
\|u\|_q^{2/\theta} \leq  C^{2/\theta}\|\nabla u\|_2^2 + \Phi_p(u)\ .
\end{equation}

Note that  \eqref{gnsgen} is both homogeneous of degree 2 and scale 
invariant, \eqref{gnsgen4} is homogeneous of degree 2, but {\it not} scale 
invariant, while \eqref{gnsgen5} is {\it not} homogeneous 
of any degree, but  \eqref{gnsgen5} is scale invariant in the sense that 
replacing $u(x)$ by  
$\lambda^{n/p}  (\lambda x)$ in (\ref{gnsgen5}), both sides scale with 
the 
same power of $\lambda$.  

At the expense of giving up either homogeneity or scale invariance, we have 
arrived at a pair of inequalities in which  the right hand sides are 
convex.

Both (\ref{gnsgen4}) and (\ref{gnsgen5}) have the general form
\begin{equation}\label{nash4}
\F(u) \leq \E_1(u) + \E_2(u)\ 
\end{equation}
where $\F$, $\E_1$ and $\E_2$ are proper lower semicontinuous convex 
functionals on a Banach space $X$.
We are thus brought to the point of needing and computing the Legendre 
transform of a 
{\it sum } of convex functions, and this leads us to the {\it infimal 
convolution,} as we now recall.

To this end, let $Y$ be another Banach space, and 
let $\langle \cdot, \cdot  \rangle$ be a bilinear form on $X\times Y$ 
making $X$ and $Y$ a dual pair of Banach spaces. For example, we might have 
$X = 
L^p(\R^n)$ and $Y = L^{p'}(\R^n)$ with the usual pairing. The Legendre 
transform of $\E_1+\E_2$ is given by \cite{R}
\begin{equation}\label{infconv}
 (\E_1+\E_2)^* (v)    =:\E_1^* \square \E_2^*(v)  = \inf_{v_1}\{ 
\E_1^*(v-v_1) + \E_2^*(v_1)\},
\end{equation}
 for all $v\in Y$, where the right hand side is {\em the infimal 
convolution} of $\E_1^* $ and  $\E_2^*$.

Then, for any $v \in Y$,  $\langle u, v\rangle - \F(u) \geq \langle u, v 
\rangle 
- (\E_1(u) +\E_2(u) )$.   The supremum over $u \in X$ yields 
\begin{equation}\label{dualform}
\F^*(v) \geq (\E_1^* \square \E_2^*)(v)\ ,
\end{equation}
which is the  inequality dual to  \eqref{nash4}.

The inequality (\ref{dualform}) is likely to be much less useful, or even intelligible, than (\ref{nash4}) unless 
one can explicitly compute a closed form of  the  infimal convolution on the 
right side of (\ref{dualform}), or at least show that a 
minimizing $v_1$ exists in (\ref{infconv}), and provide a 
characterization of this minimizing $v_1$. This latter is what we shall do 
for the Nash inequality.

In order to do this for GNS inequalities for $1<p\leq 2$,
it is often useful to make a judicious change of variable and choice of the 
dual pair $(X,Y,\langle \cdot, \cdot\rangle)$.  For example, consider the 
case in which $p=2$. Let $u\in L^2(\R^n)$ be non-negative, and let $\rho(x) 
= u^2(x)$, so that $\rho$ is a density; i.e., a non-negative integrable 
function.  Note that  
\begin{equation}\label{fisher}
\|\nabla u\|_2^2 = \frac14 \int_{\R^n}\frac{|\nabla \rho|^2}{\rho}{\rm d}x =: \frac14 I(\rho)\ .
\end{equation}
The function $I(\rho)$, often called the {\em Fisher Information} when $\rho$ is a probability density, is evidently convex in $\rho$. Thus   (\ref{gnsgen5}) may be rewritten  as
\begin{equation}\label{gnsgen6}
\|\rho\|_{q/2}^{1/\theta}  \leq  \frac{C^{2/\theta}}{4}I(\rho)  + \Phi_1(\rho)\ .
\end{equation}
Take $X = L^1(\R^n)$ and $Y = L^\infty(\R^n)$ with the usual dual pairing. Then for $V\in L^\infty$, since $I(\rho)$ is homogeneous of degree one, 
\begin{eqnarray*}
&\phantom{=}&\sup_{\rho\in L^1, \rho \geq 0} \int_{\R^n}V(x)\rho(x){\rm d}x - \frac{C^{2/\theta}}{4}I(\rho)  + \Phi_1(\rho)\\
 &=&
\sup_{\rho\in L^1, \rho \geq 0}\left\{  \int_{\R^n}V(x)\rho(x){\rm d}x - \frac{C^{2/\theta}}{4}I(\rho) \ :\ \|\rho\|_1 = 1\right\}\nonumber\\
&=& \sup_{u\in L^2} \left\{  \int_{\R^n}V(x)u^2(x){\rm d}x - {C^{2/\theta}}\|\nabla u\|_2^2 \ :\ \|u\|_2 = 1\right\}
= \lambda_1(V)\ ,
\end{eqnarray*}
where $\lambda_1(V)$ is the largest eigenvalue of the operator 
$C^{2/\theta}\Delta + V$.

Therefore, if we define ${\displaystyle \E_1(\rho) = 
\frac{C^{2/\theta}}{4}I(\rho)}$ and $\E_2(\rho) = \Phi_1(\rho)$
where $I(\rho)$ is the Fisher information defined in (\ref{fisher}) and 
$\Phi_1$ is defined in (\ref{phidef}),
the calculation made just above says that 
\begin{equation}\label{infconv2}
\E_1^*\square \E_2^*(V) = \lambda_1(V)\ ,
\end{equation}

Therefore, the dual form of (\ref{gnsgen5})
with $p=2$, which is equivalent to the $p=2$ case of (\ref{gnsgen}), 
gives a sharp bound on $\lambda_1(V)$ in terms of $\|V\|_{q/(q-2)}$. The duality was 
essential in the proof of stability bounds for this sharp bound  given 
by Carlen, Frank and Lieb \cite{CFL}.

One can make a similar change of variables for $1 < p < 2$.  In this case, for $u\in L^p$, $u\geq 0$, define $\rho = u^p$. Then
$$\|u\|_2^2 = \frac{1}{p^2}\int_{\R^n} \frac{|\nabla \rho|^2}{\rho^{(p-1)(2/p)}}{\rm d} x$$
which is again convex in $\rho$, but it is not homogeneous of degree one, and no familiar functional emerges from the infimal convolution. 

Before turning to the case $p=1$, $q=2$ of (\ref{gnsgen}), which is the Nash inequality, consider the other way, given in (\ref{gnsgen4}), of writing (\ref{gnsgen}) in the form (\ref{nash4}). Then 
$$\F(u) = \|u\|_q^2\ , \quad\E_1(u) =   (1-\theta) C^{2/(1-\theta)}\|u\|_2^2 \quad{\rm and}\quad \E_2(u) =  \theta \|\nabla u\|_2^2\ .$$
Therefore, 
$$\E_1(u) + \E_2(u) = \langle u, [-\theta \Delta + (1-\theta) C^{2/(1-\theta)}\one]u\rangle_{L^2}\ .$$
In this case, we may take $X = Y = L^2(\R^n)$ with the usual pairing, the dual inequality  is
$$\|v\|_{q'}^2 \geq \langle v, [-\theta \Delta + (1-\theta) C^{2/(1-\theta)}\one]^{-1} v\rangle\ ,
$$
and thus  the  norm of the operator
$ [-\theta \Delta + (1-\theta) C^{2/(1-\theta)}\one]^{-1}$ from  $L^{q'}(\R^n)$ to $L^q(\R^n)$, $1/q + 1/q'= 1$,  is 
exactly one. 
(This gives an alternate characterization of the best constant in the 
 family  of GNS inequalities ({\ref{gnsgen}) when $p=2$.

\section{A dual form of the Nash inequality}\label{sec2}

We now turn to the case $p=1$, $q=2$ of (\ref{gnsgen}) so that $\theta = n/(n+2)$. This case of (\ref{gnsgen}) is often  written as
\begin{equation}\label{nash1}
\|f\|_2^{2+\frac4n} \leq C_n \|\nabla f\|_2^2 \|f \|_1^{\frac4n}\ .
\end{equation}
We take the equivalent form  (\ref{gnsgen5}) as the starting point for our computation of a dual form.   In computing the dual form, 
we shall have to compute
\begin{equation}\label{nash1b}
\sup_{f\in X} \langle g,f\rangle - \frac12 \|\nabla f\|_2^2\   
\end{equation}
for $g\in Y$ with an appropriate choice of the dual pair $(X,Y, \langle 
\cdot,\rangle)$. 
The supremum in (\ref{nash1b}), taken over $f\in C_0^\infty(\R^n)$, defines 
the  the $H^{-1}$ norm 
\cite{Lax} of a locally integrable function $g$:
$$\frac12 \|g\|_{H^{-1}}^2 = \sup_{f\in C_0^\infty(\R^n)} \int_{\R^n} gf {\rm d}x - \frac12 \|\nabla f\|_2^2\ .$$
In dimensions $1$ and $2$, when $g$ is integrable, $\|g\|_{H^{-1}} = \infty$ unless $\int_{\R^n}g{\rm d}x = 0$. To see this for $n=1$, take $R>0$, and 
$f_R(x) =\sqrt{R/2}$ for $|x|< R$, $f_R(x) = \sqrt{R/2}(1 - |x|/R)$ for $R \leq |x|\leq 2R$ and $f(x) = 0$ for $|x|> R$. Then for all $R$, $\|\nabla f_R\|_2 = 1$,
while
$$\lim_{\R\to \infty} \sqrt{\frac{2}{R}}\int_{\R} g f_R {\rm d}x =  \int_{\R} g  {\rm d}x\ .$$
Thus $\|g\|_{H^{-1}} = \infty$.  A similar argument may be made for $n=2$ taking  $f_R$ to be a multiple of  $1-  |\log(|x|/R)|^\alpha$, 
$\alpha\in (0,1/2)$,  for $R \leq |x| \leq eR$.

If we take $X= Y = L^2(\R^n)$ with the usual dual pairing, then the supremum in (\ref{nash1b}) is $\frac12   \|(-\Delta)^{-1/2} g\|_2^2$, where in dimension
$n\geq 2$, $(-\Delta)^{-1/2} g$ is a multiple of $\int_{\R^m} |x-y|^{1-n}g(y){\rm d}y$. We use the notation $\|(-\Delta)^{-1/2} g\|_2$ for the $H^{-1}$ norm in all dimensions. 

 In the variational calculations that follow, it is necessary to have the identity   $\|(-\Delta)^{-1/2} g\|_2^2  = \langle g, (-\Delta)^{-1} g\rangle$.  The 
Green's function for the Laplacian on $\R^n$ is $G_n(x,y) = K_n|x-y|^{2-n}$ for $n\neq 2$, and is $G_2(x,y) = K_n\log(|x-y|)$ for constants $K_n$ 
depending only on $n$.
Therefore, we must  choose the dual pair $(X,Y, \langle\cdot,\cdot\rangle)$ so that for all $x$, $G_n(x,y)g(y)$ is integrable in 
$y$. By the HLS inequality, for $n\geq 3$, we may take $Y = L^{2n/(n+2)}(\R^n)$ and  $X = L^{2n/(n-2)}(\R^n)$ with the usual dual pairing.

For $n=1$, there is no singularity in $G_1(x,y)$ at $y=x$, but instead there is  linear growth as $|x-y|\to \infty$. Therefore we let $\mu_1$ 
denote the weighted measure ${\rm d}\mu_1(x) = (1+|x|){\rm d}x$, and take $Y = L^1(\R,\mu_1)$. With the standard dual pairing, 
we may take $X$ to consist of all measurable functions $f$ such that ${\rm 
ess sup}\{ |f(x)|/(1+|x|)\} < \infty$.  For $n=2$, $G_2(x,y)$ 
has both a logarithmic singularity at $x=y$ and logarithmic growth as $|x-y|\to \infty$.   
Let   $\mu_2$   denote the weighted measure 
${\rm d}\mu_2(x) = \log (e+|x|^2){\rm d}x$, and let $Y$ 
denote the Orlicz space $L\log L(\R^2,\mu_2)$, and let $X$ be its standard 
dual.  Summarizing,
\begin{equation}\label{Ydef}
Y = \begin{cases}   L^{2n/(n+2)}(\R^n)  & n\geq 3\\
L\log L(\R^2,\mu_2) & n=2\\
L^1(\R,\mu_1) & n=1\end{cases}
\end{equation}
where ${\rm d}\mu_1(x) = (1+|x|){\rm d}x$ and  ${\rm d}\mu_2(x) = \log (e+|x|^2){\rm d}x$

Then for all $g\in Y$, 
\begin{equation}\label{ptoential}
\varphi_g(x) := \int_{\R^n}G_n(x,y)g(y){\rm d}y 
\end{equation}
is well-defined and continuous; we refer to $\varphi_g$ as the {\em  potential of $g$}.

With the dual pairs chosen, we define
convex functionals on $X$ as follows:
\begin{equation}\label{nash3}
\F(f) = \frac12 C_n^{-1}\|f\|_2^{2+\frac4n} \ ,\qquad \E_1(f) =  \frac12\|\nabla f\|_2^2\ ,\quad \E_2(f) = \begin{cases} 0 & \|f\|_1 \leq 1\\ \infty & \|f\|_1 > 1\ .\end{cases}\ .
\end{equation}
(The functionals are defined to be $+\infty$ when the integrals defining the norms in their definition are infinite.)
We may then rewrite (\ref{nash1}) as
\begin{equation}\label{nash4b}
\F(f) \leq \E_1(f) + \E_2(f)\ .
\end{equation}
Upon computing the Legendre transforms, the equivalent {dual form of 
the Nash inequality is
\begin{equation}\label{nash4b}
\F^*(g) \geq \E_1^*\square \E_2^*(g)\ .
\end{equation}
By direct computation, $\F^*(g) = L_n\|g\|_2^{\frac{2n+4}{n+4}}$,
where $L_n$ is a constant that may be evaluated in terms of $C_n$, but see (\ref{lndef}) below. 

By our choice of the dual pairs,   we have that }
$\E_1^*(g) = \frac12 \langle g, (-\Delta)^{-1} g\rangle_{L^2}$, and
$\E_2^*(g) = \|g\|_\infty$.   Define 
${\mathcal G}(g) := \E_1^*\square \E_2^*(g)$ so that, more explicitly,
\begin{equation}\label{Gdef}
{\mathcal G}(g)  := \inf_{ h\in L^{2n/(n+2}(\R^n)}\left\{  
\frac12 \| (-\Delta)^{-1/2}(g-h)\|^2_{2} +  \|h\|_{\infty}  
\ \right\} .
\end{equation}
Thus we obtain the  following  {\bf dual version of the Nash inequality:}

\begin{thm} \label{main1} 
 For all $g\in L^2(\R^n)$, 
 \begin{equation}\label{Gdef2}
 L_n\|g\|_2^{\frac{2n+4}{n+4}} \geq {\mathcal G}(g)
 \end{equation}
 where ${\mathcal G}(g)$ is given by \eqref{Gdef}
and
\begin{equation}\label{lndef}
 L_n = \frac12 \left( {\frac{1}{\mu_1} \|\widehat g\|_2^2 + 
\langle \widehat g 
\rangle}\right){\|\widehat g\|_2^{-2(n+2)/(n+4)}},
\end{equation}
and where $\mu_1 >0$ is the principal eigenvalue of 
the (negative) Neumann Laplacian on the unit ball $B\subset \R^n$.
and  $\widehat g $ is the  positive radial monotone decreasing 
function 
supported on  $B$, specified as 
follows.

\smallskip\noindent
(1) $\widehat g-\langle\widehat  g \rangle $ is the principal eigenfunction 
of the Neumann Laplacian on $B$.

\smallskip\noindent
(2) $\int_{\R^n} (-\Delta)^{-1} (\widehat g-\langle \widehat g \rangle)dx 
=1$.

\smallskip\noindent
A non-negative function $g$ satisfies \eqref{nash8A} with equality if and 
only if, for some $x_0 \in \R^n$ and $\lambda >0$ 
\begin{equation}
 g(x) = \lambda^{n+2} \widehat g (\lambda x).
\end{equation}
\end{thm}

Theorem \ref{main1} is proved in Section \ref{sec5}, which provides 
further information on the optimizers.

The first step is to prove the existence and uniqueness of an optimizing 
$h$ in the functional ${\mathcal G}$ defined in \eqref{Gdef}

 \eqref{nash8A}.
It will be convenient to split this optimization into two steps: 
For $c> 0$, define
\begin{equation}\label{nash9}
A_g(c)  = \inf_{h\in L^2(\R^n)}\left\{  \frac12 \|  
(-\Delta)^{-1/2}(g-h)\|^2_{2} \ :\  \|h\|_\infty \leq c \ \right\} .
\end{equation}
Then (\ref{nash8A}) can be written as
\begin{equation}\label{nash8B}
L_n\|g\|_2^{\frac{2n+4}{n+4}}  \geq \inf_{ c>0}\left\{   c +A_g(c) \ \right\} \ .
\end{equation}
The right side of (\ref{nash8B}) is the infimal  convolution functional that  we study in the
 next section,
with the goal of  rendering this functional more explicit by showing that an optimal $h$  
in (\ref{nash9}) and $c$ in (\ref{nash8B}) exist, and by giving  a
reasonably explicit 
determination of them in terms of $g$. Our main results are contained in 
Theorems~\ref{thm2} and Theorem~\ref{main2}
The information provided by Theorem~\ref{main2} on the optimizing $h$
is explicit enough that it is the basis of a direct determination of the 
sharp constant $L_n$ and all of the optimizers for (\ref{nash9}), 
and, hence, {\it for the original Nash inequality!} We also obtain  
a new weighted generalization, proved in Theorem~\ref{thm3}. 

\begin{remark}\label{scaleinv}
The scale  invariance  of the inequality (\ref{nash8A}) will be useful later on. 
Let $\lambda > 0$ be a scaling parameter and, given $f\in Y$, define
\begin{equation}\label{scaleinv1}
f_\lambda(x) = \lambda^{n+2} f(\lambda x)\ .
\end{equation}
One readily  checks that
$$\\| (-\Delta)^{-1/2}(f_\lambda - h_\lambda)\|_2^2  = \lambda^{n+2} 
\|(-\Delta)^{-1/2}(f - h)\|_2^2
\quad{\rm and}\quad \|h_\lambda\|_\infty  = \lambda^{n+2}\|h\|_\infty\ ,$$
so that the two terms in the infimal convolution scale with the same power of $\lambda$. 
Another simple computation shows that 
$$\left(\|g_\lambda\|_2^2\right)^{(n+2)/(n+4)} = \lambda^{n+2} \left(\|g\|_2^2\right)^{(n+2)/(n+4)}\ .
$$
Therefore, if $g$ is an optimizer for   (\ref{nash8A}), and if, for this $g$, 
there exists an optimal $h$ in the  infimal convolution in 
 (\ref{nash8A}), then for all $\lambda >0$, $g_\lambda$ is an 
optimizer 
for (\ref{nash8A})  and $h_\lambda$ is an optimizer in the corresponding 
infimal convolution, in accordance with the final statement in Theorem~\ref{main1}. \hfill$\square$
 \end{remark}

We close this section by giving a   direct proof of (\ref{nash8A}) with 
a sub-optimal constant, and introducing some notation that will be used in 
the sequel. Part of this proof 
will be useful in the next section when we show that  an optimizing $h$ and 
$c$ exist. We give the details for $n\geq 2$; the adaptation to $n=1$ is 
left to the reader.

By the HLS inequality, $\| (-\Delta)^{-1/2} 
(|g|-c)_+\|^2_{2} \leq C_{n,1} \|(|g|-c)_+\|^2_{2n/(n+2)}$. To 
estimate the right side in terms of
$\|g\|_2^2$ and $c$, note that 
 for any $c>0$, 
\begin{equation}\label{cut}
\int_{\R^n} (|g|- c)_+^{2n/(n+2)}{\rm d} x \leq   \int_{\R^n}  
|g|^{2n/(n+2)} \mathbb{1}_{\{|g| \geq c\}}{\rm d} x \leq S
\|g\|_2^{n/(n+2)}\left(|\{ |g| \geq c\}|\right)^{2/(n+2)}\ ,
\end{equation}
where $|\{ |g|\geq c\}|$ denote the measure of the set $\{ x\ : |g|(x) \geq 
c\}$.
By Chebychev's inequality, $|\{ |g| \geq c\}| \leq \|g\|_2^2 c^{-2}$, and 
hence 
$$\|(|g|-c)_+\|_{2n/(n+2)}^2 \leq (\|g\|_2^2)^{\frac{n+2}{n}} c^{-4/n}\ .$$
For the set on which $|g|\geq c$ we take $h 
= c \, {\rm sgn}\, (g)$.
For the set on which $|g|<c$ we can take $h=g$, whence $g-h=0$. 
Altogether, with this choice of $h$, 
${\displaystyle A_g(c) \leq C_{n,1}  (\|g\|_2^2)^{\frac{n+2}{n}} c^{-4/n}}$.
and consequently, 
$$
\inf_{c>0} \left\{  c+   A_g(c)  \ \right\}   \leq \inf_{c>0}\left\{  c+  C_{n,1}  \|g\|_2^{n+2} c^{-4/n}  \ \right\}  = K_n  \|g\|_2^{\frac{2n+4}{n+4}}
$$
for an explicitly computable constant $K_n$ depending only on $n$ and $C_{n,1}$.   For $n=1,2$, the contruction must be modified to ensure that $\int_{\R^n}(g- h){\rm d}x =0$; this is easly done. 
By duality, this gives a proof of the Nash inequality via the HLS inequality, but not with the sharp constant. The results of the next section provide more information.

\section{Optimizers for the infimal convolution}\label{sec3}

Our goal in this section is to show the existence of optimizers  $h$ 
{\em for a fixed } $g$ in the infimal convolutions functional 
\begin{equation}\label{nash8X}
{\mathcal G}(g) = \inf_{ h\in Y}\left\{  
\frac12 \| (-\Delta)^{-1/2}(g-h)\|_2^2 +  \|h\|_\infty  
\ \right\} \ .
\end{equation}
and determine their form as explicitly as possible. Our main results are Theorem~\ref{thm2} and Theorem~\ref{main2}.

\begin{lm}[\bf Existence and uniqueness of minimizers of 
$A_{g}(c)$]\label{thm1} Let  $g\in Y$, with $Y$ defined in (\ref{Ydef}), 
let $c\in (0,\infty)$, and let $A_g(c)$ be defined as in (\ref{nash9}). 
Then there exists a unique $L^{2n/(n+2)}(\R^n)$ function $h$ 
satisfying $\|h\|_\infty \leq c$ and
$$A_g(c) = \frac12 \| (-\Delta)^{-1/2}(g-h)\|_2^2\ .$$
\end{lm}

\begin{proof}  Fix $g\in Y$. Let $\mathcal{K}$ denote the set  of functions $f\in H^{-1}(\R^n)$ such that $h := g-f$ satisfies $\|h\|_\infty \leq c$. This  set
$\mathcal{K}$ is non-empty  since, for $n\geq 3$, it contains $f := g - \max\{\min\{g(x),c\},-c\}$. For $n=1,2$, a small modification is needed to ensure that
$\inf_{\R^n}f{\rm d}x = 0$, and that $f\in H^{-1}(\R^n)$: One may increase 
the size of the set on which $|h| = c$. 

The set $\mathcal{K}$ is evidently convex. To see that it is closed in 
$H^{-1}(\R^n)$, let 
$\rho$ be a positive, compactly supported $C^\infty$ function with $\int_{\R^n}\rho{\rm d}x =1$.  Then $(-\Delta)\rho(x) \in H^{-1}(\R^n)$ (with $\|(-\Delta)\rho\|_{H^{-1}} = \|\nabla \rho\|_2$).
Let $\{f_n\}$ be any convergent sequence in $\mathcal{K}$ with limit $f\in H^{-1}(\R^n)$.  Then
\begin{multline*}
\int_{\R^d} \rho(x) [g(x) - f_n(x)]{\rm }x = \langle (-\Delta)\rho, g-f_n\rangle_{H^{-1}}  \to\\   \langle (-\Delta)\rho, g-f \rangle_{H^{-1}}  =  \int_{\R^d} \rho(x) [g(x) - f(x)]{\rm }x\ .\end{multline*}
It follows that, with $h = g-f$ and $h_n = g-f_n$,
${\displaystyle  \int_{\R^d} \rho h {\rm d}x =  \lim_{n\to\infty} \int_{\R^d} \rho h_n {\rm d}x \leq c}$
Since this is true for all such $\rho$,  $\|h\|_\infty \leq c$.    Since  $H^{-1}(\R^n)$ is a Hilbert space, 
the  {\em projection lemma} \cite[Lemma  2.8]{LL}, says that the closed, non-empty convex set $\mathcal{K}$ contains a unique element $f = g-h$ of minimal norm. 
\end{proof}

\begin{thm}[\bf Properties of the minimizer of $A_{g}(c)$]\label{thm2}  Let 
 $g\in Y$, with $Y$ defined in (\ref{Ydef}), let $c\in (0,\infty)$, and let 
$A_g(c)$ be defined as in (\ref{nash9}), and let $h$ be the optimizer in 
(\ref{nash9}). Then:

\smallskip
\noindent{\it (1)}  For $n\geq 3$, ${\displaystyle \varphi(x) := (-\Delta)^{-1}(g-h)(x)= 0}$ for all $x$ such that $|h(x)| < c$.   For $n=1,2$, 
$\varphi$ is constant on this set. 

\smallskip
\noindent{\it (2)} On the set $\{ x\ : |h(x)| < c\}$, $h(x) = g(x)$.  Furthermore $g-h$ is integrable and $\int_{\R^n}(g-h){\rm d}x = 0$.

\smallskip

Under the  further assumption that $g$ is non-negative,

\smallskip
\noindent{\it (3)}  The function $h$ is non-negative, and on the set $\{ x\ : g(x)\geq c\}$, $h(x) = c$. 
\end{thm}

\begin{proof} For $\tilde h\in Y$, define
${\displaystyle {\mathcal G}(\tilde h) =  \frac12 \langle g-\tilde h \, ,\,  (-\Delta)^{-1}g-\tilde h\rangle_{L^2}}$.
Then
$$\varphi(x) = (-\Delta)^{-1}(h-g)(x) = \frac{\delta{\mathcal{G}}}{\delta h}(x)\ ,$$
the first variation of ${\mathcal G}$ at $\tilde h= h$.   On the set $\{x\:\ 
|h(x)| < c\}$, the first variation must vanish for $n\geq 3$.  If $n=1.2$ 
the only allowed variations $\delta h$ in $h$ are those with R$\int_{\R^n} 
\delta h {\rm d}x =0$.  In this case we conclude that $\varphi$ is constant 
on the set  $\{x\:\ |h(x)| < c\}$.
This proves {\it (1)}. 

By  part {\it (1)} and the  `No Flat-Spots'  Theorem  of Frank and Lieb \cite{FL}
$\Delta \varphi = 0$ on the set $\{x\ :\ \varphi(x) = 0\}$. This implies that $g= h$ on this set, 
which by {\it (1)}, 
includes the set $\{x\:\ |h(x)| < c\}$.  Since $h\in L^{2n/(n+2)}$, the set  $\{x\:\ |h(x)| = c\}$ 
has finite measure, and thus $g- h$ is integrable. 
If the integral were not zero, it would be impossible for $\varphi$ to 
vanish outside a set of finite measure.  This proves {\it (2)}. 

For the rest, suppose that $g \geq 0$, and  relax the variational problem for $A_g(c)$: Minimize the functional ${\mathcal G}$ over all $h$ satisfying only $h(x) \leq c$, and not $|h(x)| \leq c$. 

On the set where the minimizing $h$ satisfies $h(x) < c$, it must be the case that $\varphi(x) = 0$, and 
then as before, $h=g$ on this set. Since $g\geq 0$, it is  possible that $h$ takes on negative 
values. Thus $h \geq 0$, the the constraint $h \leq c$ has the same force as the constraint $|h| \leq c$.  Since on the set $\{ x\ : g(x)\geq c\}$, $h= g$ is impossible, it must be that $h(x) = \pm c$ for almost every $x$ in this set. Since $h$ is non-negative, this proves {\it (3)}. 
\end{proof}

\begin{remark}[\bf On the positivity assumption in {\it (3)} of 
Theorem~\ref{thm2}]  If $g$ takes on both signs, 
 the optimizing $h$ {\em need not} satisfy $h = c$ on $\{g \geq c\}$ and $h 
= -c$ on $\{g \leq -c\}$, because the potential $\varphi$ is depends on the 
values of $g$ and $h$ in a non-local manner. 
For example, consider  $0 < c< 1$, and
$$
g(x) = 
\begin{cases}
\phantom{-}0, & |x|> 1\\
-a, &  1 \geq |x| > 1/2\\
\phantom{-}1, & 1/2 \geq |x|
\end{cases}
$$
where $a$ is a large positive number.  Then the potential $\varphi_g$ will be non-positive everywhere, which means that $h(x) = c$  on the set where
$c < |g(x)|$, which is the entire support of $g$. 
\end{remark}

We now turn to the determination of the optimal $c$ in (\ref{nash8B}).  For any  $c \geq 0$, let $\varphi_{(c)}$ denote the potential $\varphi_{(c)} = (-\Delta)^{-1} (g- h_c)$. 

\begin{lm}[\bf Properties of $A_g(c)$]\label{strictconv} For all $g\in Y$, 
$A_g(c)$ is a strictly 
convex function of $c$ on the interval $(0,\|g\|_\infty)$.
Moreover, the map $c\mapsto  \varphi_{(c)}$ is continuous into $L^1(\R^n)$ on $(0,\|g\|_\infty)$.
\end{lm}

\begin{proof}  Let $0 < c_0 < c_1 < \|g\|_\infty$, and let $\lambda \in 
(0,1)$. Define $c_\lambda = (1-\lambda)c_0 + \lambda c_1$.  For all $c\in 
(0,\|g\|_\infty)$, let $h_c$ denote the optimizing $h$ in the variational problem defining $A_g(c)$. Since $|h_{c_1}| = c_1$ on the set where $|g| > c_1$, while $\|h_{c_0}\|_\infty = c_0 < c_1$, $h_{c_0} \neq h_{c_1}$. 

Note that
$$\|(1-\lambda)h_{c_0} + \lambda h_{c_1}\|_\infty \leq (1-\lambda)c_0 + \lambda c_1 = c_\lambda\ ,$$
and, therefore, by the  strict convexity of $h \mapsto \langle g - h, 
(-\Delta)^{-1}(g-h)\rangle$, 
\begin{eqnarray}A_g(c_\lambda) &\leq& 
\frac12 \langle g-[(1-\lambda)h_{c_0} + \lambda h_{c_1}] \, ,\,  (-\Delta)^{-1}(g-[(1-\lambda)h_{c_0} + \lambda h_{c_1}])\rangle_{L^2}\nonumber\\
&<& 
(1-\lambda) \frac12 \langle g- h_{c_0}  \, ,\,  (-\Delta)^{-1}(g-h_{c_0} )\rangle_{L^2}  +
\lambda \frac12 \langle g- h_{c_1}  \, ,\,  (-\Delta)^{-1}(g-h_{c_1} )\rangle_{L^2} \nonumber\\
&=& (1-\lambda)A_g(c_0)  + \lambda A_g(c_1)\ .
\nonumber
\end{eqnarray}
Next, by the parallelogram law in the Hilbert space $H^{-1}(\R^n)$, and by
the convexity proved above,
$$ \| (-\Delta)^{-1/2}(h_{c_0} - h_{c_1})\|_2^2  \leq A_g(c_0) + A_g(c_1) - A_g(c_{1/2})\ .$$  
Since $A_g$ is convex, it is continuous, and, as $c_1$ and $c_0$ approach 
each other, 
$ \| (-\Delta)^{-1/2}(h_{c_0} - h_{c_1})\|_2^2$ tends to zero. This implies that $c\mapsto (-\Delta)^{-1/2}h_c$ is continuous into
$L^2$, and then, for $n\geq 3$,  by the HLS inequality, $c\mapsto \varphi_{(c)}$ is continuous into $L^{2n/(n-2)}$. 
Since for any $0 < c < c_0$, $\varphi_{(c_0)} - \varphi_{(c_1)}$ vanishes outside the set on which $h_c  =c$, and since this set has finite measure, the continuity into $L^1(\R^n)$ follows from the continuity into $L^{2n/(n-2)}(\R^n)$. 
\end{proof}

The strict convexity proved in Lemma~\ref{strictconv} shows that {\it there is a unique minimizing $c$ in (\ref{nash8B})}. 
To obtain an equation for this minimizing $c$, we proceed under the 
assumption that $g\geq 0$.

Define 
$E_c := \{ x\ :\ \varphi_{(c)}(x)>0\}$ and define $F_c = \{ x\ :\ h_c(x) = c\}$.  On the set $E_c$, $h_c = c$ 
since, otherwise, one could lower the value of the functional defining 
$A_g(c)$ by increasing $h$. 
By the optimality of $h_c$, $\varphi_{c} =0 $ on the 
complement of $F_c$. That is, $E_c = F_c$, which shows that $F_c$ is open. 

Note that for $c_0 < c_1$, $\varphi_{c_1} \leq  \varphi_{c_1}$ and $E_{c_1} \subset E_{c_2}$. 
By Lemma~\ref{strictconv}, the function $c\mapsto A_g(c)$ has left and 
right derivatives at each $c\in (0,\|g\|_\infty)$. 

\begin{lm}\label{difflem} For all $0 < c < \|g\|_\infty$, 
\begin{equation}\label{elc}
\lim_{s \downarrow 0}\frac{A_g(c) - A_g(c-s)}{s} \leq \int_{\R^n} \varphi_{(c)}{\rm d}x \leq 
\lim_{s \downarrow 0}\frac{A_g(c+s) - A_g(c)}{s}\ .
\end{equation}
\end{lm}

\begin{proof} 
Let $0 < c < \|g\|_\infty$ and let $s>0$ be small enough so that $c+s < 
\|g\|_\infty$.  Then
\begin{eqnarray}
A_g(c+s) - A_g(c) &=& \frac12\langle g-h_{c+s}, \varphi_{(c+s)}\rangle -  \frac12\langle g-h_{c}, \varphi_{(c)}\rangle\nonumber\\
&=& \frac12\langle h_c-h_{c+s}, \varphi_{(c+s)}\rangle 
+ \frac12\langle g-h_{c}, \varphi_{(c+s)} -\varphi_{(c)}\rangle\ . \nonumber
\end{eqnarray}
By what has been explained above, $ h_c-h_{c+s} =-s$ on the set where $\varphi_{(c+s)}$ is non-zero,  Furthermore,
\begin{eqnarray}
\langle g-h_{c}, \varphi_{(c+s)} -\varphi_{(c)}\rangle &=&  \langle \varphi_{(c)} , h_c - h_{c+s}\rangle\nonumber\\
&=&  \langle \varphi_{(c+s)} , h_c - h_{c+s}\rangle\nonumber +  \|(-\Delta)^{-1/2}(h_c - h_{c+s})\|_2^2\nonumber
\end{eqnarray}
Altogether,
\begin{equation}\label{leftder}
A_g(c+s) - A_g(c)  = -s \int_{\R^n}\varphi_{c+s}{\rm d}x + \frac12 \|(-\Delta)^{-1/2}(h_c - h_{c+s})\|_2^2\ .
\end{equation}

Similar computations show that for $s>0$ and $c-s>0$,
\begin{equation}\label{rightder}
A_g(c) - A_g(c-s)  = -s \int_{\R^n}\varphi_{c}{\rm d}x - \frac12 \|(-\Delta)^{-1/2}(h_c - h_{c+s})\|_2^2\ .
\end{equation}
Now (\ref{elc}) follows from the monotonicity of $\varphi_{c}$. 
\end{proof}

\begin{remark}\label{diflem}  Lemma~\ref{difflem} shows that if $A_g$ is 
differentiable at $c$, then $A'_g(c) = -\int_{\R^n}\varphi_{c}{\rm d}x$;
by the convexity proved in Lemma~\ref{strictconv}, $A_g$ is 
differentiable almost everywhere.  Moreover, by the second part of
Lemma~\ref{strictconv}, $c\mapsto -\int_{\R^n}\varphi_{c}{\rm d}x$ is 
continuous.
\end{remark}

We now turn to the minimization problem in (\ref{nash8B}):

\begin{lm}[\bf Equation for $c$]\label{optc} For $g\in Y$, $g\geq 0$, here 
exists a  unique $c_0 \in (0,\|g\|_\infty)$ such that
\begin{equation}\label{gcond2}
\int_{\R^n} \varphi_{(c)}{\rm d}x =1\ ,
\end{equation}
and
$$c_0 + A_g(c_0) < c+ A_g(c)$$
for all $c \neq c_0$ in $(0,\|g\|_\infty)$. 
\end{lm}

\begin{proof} Since $g\geq 0$, $\lim_{c\to 0}\int_{\R^n}\varphi_{(c)}{\rm 
d}x  = \int_{\R^n}\varphi_g(y){\rm d}y = \infty$. It is also clear that
$\lim_{c\to \|g\|_\infty} \int_{\R^n}\varphi_{(c)}{\rm d}x = 0$.  Lemma~\ref{strictconv} and  Lemma~\ref{difflem} show that the 
 function $c\mapsto 
\int_{\R^n}\varphi_{(c)}{\rm d}x$ is 
continuous and strictly monotone.   It follows that there exists a unique
$c_0 \in (0,\|g\|_\infty)$ such that 
\begin{equation}\label{gcond2}
\int_{\R^n} \varphi_{(c_0)}{\rm d}x =1\ .
\end{equation}
Then by Lemma~\ref{diflem},  the left derivative of $c+ A_g(c)$ is
non-positive at $c=c_0$, while the right derivative  is non-negative.  It follows that in this case
$c_0$ minimizes $c+ A_g(c)$. 
\end{proof}

In summary, we have proved  the following:

\begin{thm}[\bf Existence, uniqueness and propeties of 
optimizers of  $\mathcal{G}$]\label{main2} Let  $g\in Y$, $g\geq 0$. Then 
there is a unique $h\in Y$ such that for all $\widetilde h\in Y$,  
$\widetilde h \neq h$, 
$${\mathcal G}(g) =  \frac12 \| (-\Delta)^{-1/2}(g-h)\|_2^2 + \|h\|_\infty  <  \frac12 \| (-\Delta)^{-1/2}(g-\widetilde h)\|_2^2 + \|\widetilde h\|_\infty\ .$$
Moreover:

\smallskip
\noindent{\it (1)}  $h\geq 0$ and  $g-h$ is integrable with $\int_{\R^n}(g-h){\rm d}x = 0$.

\smallskip
\noindent{\it (2)}On the set $\{x\ : \ h(x) \neq \|h\|_\infty$, $h(x) = g$. and 
$\varphi_{g-h}(x) = 0$.

\smallskip
\noindent{\it (3)}  $\int_{\R^n}\varphi_{g-h}{\rm d}x =1$. 
\end{thm}

\subsection{The case in which $g$ is radially symmetric and 
monotone}\label{sec4} 

Now consider the case in which $g$ is non-negative, radially symmetric and 
monotone; i.e., the case in which $g = g^*$, its radially symmetric 
decreasing rearrangement. In this case, Theorem~\ref{main2} gives us an 
explicit description of the optimizing $h$.   For some $c \leq \|g\|_\infty$,
$h = h_c$
where
\begin{equation}\label{optimalh}
h_c(x) = \begin{cases} c & |x| \leq r(c)\\ g(x) & |x| > r(c)\ , \end{cases}
\end{equation}
and where $r(c)$ is uniquely determined (for fixed $c$) by the requirement that 
$$\int_{\R^n}(g- h_c){\rm d}x = 0\ ,$$
and then $c$ itself is fixed by the requirement that
$$\int_{|x|\leq r(c)} \int_{|y| \leq r(c)} G_n(x,y)(g(y) -c){\rm d}y{\rm d}x = 1\ .$$

\section{Proof of Theorem \ref{main1} concerning the dual \\ Nash 
inequality}\label{sec5}

In the previous sections we have determined the optimal $h$ for a given 
function $g$. 
We are now ready to determine the optimizers $g$ for the dual Nash 
inequality (\ref{nash8B})  that will determine the sharp constant $L_n$
in \eqref{nash8A}.  That is, we shall now determine
$$\sup_{g\in L^2}\frac{{\mathcal G}(g)}{\|g\|_{2}^{\frac{2n+4}{n+4}} }\ .$$

Our first goal is to reduce to the radial decreasing case, but ${\mathcal G}$ is not monotone under rearrangement, so we must make an indirect argument. 
By duality there is a one to one correspondence between optimizers for the 
Nash inequality and its dual. Moreover, the functional gradient of a 
power of the $L^2 $ norm evaluated at $f$ is a multiple of $f$. Therefore, 
every optimizer for the dual Nash inequality is a multiple of 
the Nash inequality and {\em vice-versa}. By the Faber-Krahn inequality, 
the Nash inequality has symmetric decreasing optimizers, and hence so does 
the dual Nash inequality. Once we have characterized the symmetric 
decreasing optimizers for either inequality, a standard argument, recalled 
below,  that relies on the cases of equality in the Faber-Krahn 
inequality, shows that all optimizers are symmetric decreasing. 
Consequently we are allowed to restrict attention to symmetric decreasing 
functions in our search for optimizers of the dual inequality. We emphasize 
that our proof of the sharp dual form  of the Nash inequality is 
completely independent of the the proof of  the sharp Nash 
inequality \cite{CL}; all we need to know is that minimizers of the Nash 
inequality are necessarily symmetric decreasing. 

The first thing to do is show that an optimizer exists, either for the 
Nash inequality or its dual. The existence for the Nash inequality is a 
special case of a well known result for the full family of GNS inequalities
\eqref{gnsgen}. For a simple proof  in the case $p =2$ see \cite[Lemma 
4.2]{CFL}. This proof is easily adapted to $p=1, \, q=2$, which is the 
Nash inequality. We are now ready to prove Theorem \ref{main1}. 

\begin{proof}[Proof of Theorem \ref{main1}]

Suppose that  
$g$ is a symmetric decreasing optimizer. Let $h_g$ be the corresponding 
optimizer in \eqref{nash8X}. The support of $h_g$ is a closed ball
$\overline{B(r)}$ of 
some radius $r>0$. 
Then $g(x)$ must vanish for $|x| > r$ 
since any non-zero values of $g$ in the region $|x| > r $
contribute on the left side of (\ref{nash8B}), 
but not on the right side. 
This shows that the optimizers are compactly supported. 

Since $g\in Y$ 
and has compact support, $g\in L^1(\R^n)$.  Taking the first 
variation of $L_n (\|g\|_2^2)^{(n+2)/(n+4)}$ at the optimizer $g$, we obtain 
a $Cg$ for some constant $C$ depending on $\|g\|_2$ and $n$.  In computing 
the variation of the infimal convolution, we use the fact that a 
minimizing $h_g$ exists so that the infimal convolution \eqref{nash8X} is 
equal to
 $$\frac12 \langle g- h_g, (-\Delta)^{-1}(g - h_g)\rangle - \|h_g\|\ .
 $$
By the optimality of $h_g$ there is no contribution from the variation 
of $h_g$ as $g$ is varied, and so the
variation of the infimal convolution at $g$ is  $\varphi_g =  
(-\Delta)^{-1}(g - h_g)$.  It follows that
$$Cg = (-\Delta)^{-1}(g - h_g)\ .
$$
Taking the Laplacian of both sides, 
\begin{equation}\label{elA}
C(-\Delta) g = g - h_g\ .
\end{equation}
Unless the radial derivative of $g$ vanishes on the boundary of the ball 
supporting $\varphi_g$ (and hence $g$), there is a non-zero measure 
concentrated on the boundary of this ball included in $(-\Delta)g$. Since 
there is no such measure on the right side of (\ref{elA}), $g$ satisfies 
{\em Neumann (as well as Dirichlet) boundary conditions} on $B(r)$.
 Let 
 $$\langle g\rangle = \frac{\int_{\R^n} g{\rm d}x}{|B(r)|}\ .$$
 The function $h_g$ is very simple in this case. It equals the constant
 $ \langle g\rangle$ times the indicator function of the ball 
$B(r)$. Thus, letting $\Delta_N$ be the Neumann Laplacian on 
$B(r)$, 
 \begin{equation}\label{elA}
(-\Delta_N) (g -\langle g\rangle)   = \frac{1}{C}(g - \langle g\rangle)\ .
\end{equation}
That is, $(g -\langle g\rangle)$ (and not $g$ itself) is a radial 
eigenfunction of the Neumann 
Laplacian $\Delta_N$ on $B(r)$, and $g$ is radial decreasing. This is only 
possible if 
$g -\langle g\rangle$ is the radial eigenfunction of $-\Delta_N$ with the 
least strictly positive eigenvalue, $\mu_1(r)$, and, in this case, $C = 
1/\mu_1(r)$.

Our conclusion is that if $g$ is an optimizer supported in $\overline 
{B(r)} $  then $g-\langle g \rangle $ must be 
a multiple of 
this Neumann 
eigenfunction.

Unlike the case for the direct Nash inequality, which is 
homogeneous, there is no simple scaling now and thus, for each ball with 
given radius, there is exactly one optimizer supported in that ball.

Let  $\psi_1$ be the first non-constant radial eigenfunction  of 
$-\Delta_N$ on 
$B(r)$ that satisfies $\psi_1(0) = 1$.  This is a Bessel function 
\cite{CL}. For any $a>0$, define $c= -a\psi_1(x)$ for $|x| = r$, and define 
$g = a\psi_1 + c$. Then $g$ is radially symmetric and monotone decreasing, 
and $\langle g\rangle = c$.
 Thus, $g -\langle g\rangle = a\psi_1$, and for $x\in B(r)$,
 $$ -\Delta g = (-\Delta_N) (g -\langle g\rangle)  = \mu_1(r) a\psi_1 = \mu_1(r)( g  -\langle g\rangle) \ ,$$
 for all $x$,
 $-\Delta g = \mu_1(r)( g  -h_{\langle g\rangle})$, which is the same as
 $$(-\Delta)^{-1} ( g  -h_{\langle g\rangle}) = \frac{1}{\mu_1(r)}  g\ .$$
 Therefore,
 \begin{equation}\label{elB}
 \int_{\R^n} (-\Delta)^{-1} ( g  -h_{\langle g\rangle}) {\rm d}x = \int_{\R^n} g = \frac{1}{\mu_1(r)} |B(r)| \langle g \rangle\ .
 \end{equation}
 There is a unique value of $a$ such  that the right side of (\ref{elB}) 
equals $1$. For this value of $a$, 
 $h_g$ is the optimal $h$ for $g$ in the infimal 
convolution, and 
 \begin{eqnarray}
 \inf_{ h\in Y}\left\{  \frac12 \langle  g-h \, ,\,  (-\Delta)^{-1} g-h\rangle_{L^2} +  \|h\|_\infty  \ \right\} 
 &=& \frac12 \langle  g- h_{\langle g \rangle} \, ,\,  (-\Delta)^{-1}  g-h_{\langle g \rangle}\rangle_{L^2} + 
  \|h_\langle g \rangle\|_\infty\nonumber\\
  &=&  \frac12 \frac{1}{\mu_1(r)} \|  g- \langle g \rangle\|^2_{L^2(B(r))} \rangle_{L^2} + 
 \langle g \rangle\nonumber\\
  &=&   \frac12 \frac{1}{\mu_1(r)} \|g\|_2^2 + \frac12 \langle g \rangle\ , 
 \nonumber
  \end{eqnarray}
  where  the last equality is valid since the quantity on the right side of (\ref{elB}) equals $1$.  This function $g$
  is then the unique optimizer of the dual Nash inequality supported in 
$\overline{B(r)}$, and the sharp constant
  $L_n$ in \eqref{nash8A}  is given by \eqref{lndef}.
  By the remarks on scale invariance, this is independent of $r$, as it must be. 
  
  \end{proof}

\section{Weighted Nash inequalities}

The method of 
Section~\ref{sec3} can be used to compute a more general class of infimal 
convolutions. Let $w(x)$ be a strictly positive function on $R^n$. 
Let $\E_1^*(g)= \frac12 \langle g, -\Delta^{-1} g \rangle$ and let  
$E_2^*(g) =  \Vert g/w \Vert_\infty$. Then
\begin{equation} \label{nash9B}
 \E_1^*\square \E_2^*(g) =   \inf_{ h\in Y}\left\{  
\frac12 \langle g-h \, ,\,  (-\Delta)^{-1}g-h\rangle_{L^2} +  
\|h/w \|_\infty  
\ \right\} \ .
\end{equation}

When $g$ is nonnegative a unique optimal $h$ exists, and the other theorems 
of Sections \ref{sec3} and \ref{sec4}  apply as well. In particular, 
{\em when 
$g$ is symmetric decreasing and $w$ is symmetric increasing the minimizing 
$h$ equals a constant times $w$
inside a centered ball $B$ of some radius $r$ and equals $g$ outside that 
ball}. 
The radius and the constant are determined by the condition that $\int_B h 
= \int_B g$. 

When $w(x) =|x|^{p} $, with $p>0$, the method of Section \ref{sec5} can 
be applied to 
determine the constant in the sharp inequality 
\begin{equation}\label{dwnash}
L_{n,p}\Vert g\Vert_2^{2\beta} \geq  \inf_{ h\in Y}\left\{  
\frac12 \langle g-h \, ,\,  (-\Delta)^{-1}g-h\rangle_{L^2} +  
\|\, |x|^{-p} h \|_\infty  
\ \right\} \ .
\end{equation}
For $p=0$ this reduces to the dual Nash inequality. Otherwise, we can call 
this the {\bf dual weighted Nash inequality}. 

By the arguments of Section \ref{sec2} this is equivalent to the following 
{\bf sharp weighted Nash inequality}:
\begin{equation}\label{wnash}
\Vert f \Vert^{2+\gamma} \leq C_{n,p} \Vert \nabla f\Vert_2^2
\Vert \, |x|^{p}f\Vert_1^\gamma \ .
\end{equation}
This inequality is homogeneous of order $2+\gamma$ in $f$ and scales 
correctly for 
$\gamma =n/(4+2p)$

As in Section \ref{sec5}, since the right side of \eqref{wnash} is monotone 
under symmetric decreasing rearrangement, if \eqref{dwnash} has an 
optimizer $g$ then it has a symmetric decreasing optimizer  $g$.
To keep this section short we pass over 
the proof that an optimizer exists, and leave it  to the reader. (The argument in  \cite[Lemma 
4.2]{CFL} can be adapted to this end.)

Assume $g$ is a symmetric decreasing optimizer for \eqref{dwnash}. Then it 
satisfies the Euler-Lagrange equation \eqref{elA}, in which $h_g $ is the 
optimizer in the weighted infimal convolution. (Since $h_g $ is a 
minimizer, we do not have to be concerned about the variation of $h_g$ with 
respect to $g$.)

As before, let $\varphi$ be the potential of $g-h_g$ and let 
$\overline{B(r)}$ be the  closed ball that supports $\varphi$. By scale 
invariance we shall assume that $r=1$ and just write $B$.  Then, by 
\eqref{elA}, $\overline{B}  $  also supports $g$  Since $g$ is an 
optimizer, $g$ has the same support as $\varphi$, i.e., $\overline{B}$. 
This implies that 
$h_g$ equals a constant multiple $\alpha$  of $w =|x|^p$ inside the ball 
and is zero outside. 
The constant $\alpha$  is fixed by the condition that $h_g $ and $g$ have 
the same integrals. Consequently, $h_g = \left(\langle g\rangle / \langle 
|x^p|\rangle \right) |x|^p \mathbb{1}_{B}$. By substituting this into
\eqref{elA}, we find that $g$ must satisfy the homogeneous equation
\begin{equation} \label{elB}
 Cg = (-\Delta)^{-1} \left(g-\langle g\rangle \frac{|x|^p}{
 \langle 
|x^p|\rangle } 
\mathbb{1}_{B}\right)\ = \ \varphi \ .
\end{equation}

A slight variation from Section \ref{sec5} is that the optimality of $h_g$ 
yields the following condition on the optimizer, which breaks the 
homogeneity:
\begin{equation} \label{elC}
    \int_{\R^n} \varphi(x)\, |x|^p \, dx =1\ .
\end{equation}
Among all the solutions of \eqref{elB} only one of them will satisfy
\eqref{elC}. This is the optimizer supported by $B$; scaling provides the 
rest.  

The solution of the linear equation \eqref{elB}, subject to the linear
constraint \eqref{elC}, is left to the reader. When $p$ is an even integer, 
however,  \eqref{elB} can be cast as an eigenvalue problem, as in the 
unweighted case, except that the boundary conditions is  Robin instead 
of Neumann. We explain how this is done when $p=2$, from which the reader 
will easily perceive the solution to the $p=2k$ case. 

Apply the Laplacian to both sides of \eqref{elB} and obtain
\begin{equation} \label{elD}
 -\Delta g =  \frac1C \left(g-\langle g\rangle \frac{|x|^2}{
 \langle 
|x^2|\rangle } 
\mathbb{1}_{B}\right).
\end{equation}

Next, with $ \alpha := \langle g\rangle /\langle 
|x^2|\rangle$, \  add
$$
\Delta\left( \alpha  |x|^2   -  C 2n \alpha 
\right) = 2n \alpha  
$$ 
to both sides of \eqref{elD} and obtain, with
$f:= g- \alpha \left(    |x|^2 -  C 2n 
\right),$
\begin{equation}\label{elE}
 -\Delta f = \frac1C f.
\end{equation}
This equation holds everywhere in the unit ball $B$. As in Section 
\ref{sec5},  $g$ satisfies both Dirichlet and Neumann boundary conditions on 
the boundary of $B$. Therefore, on the boundary of $B$, the ratio of the 
normal derivative of $f$ to the value of $f$ is given by $\rho := 2 / 
(1-2nC) $, 
independent of $g$. This a Robin boundary condition.
We conclude that the optimizer $g$ is such that $f$ is an 
eigenfunction of the Laplacian on the unit ball with this Robin boundary 
condition, and that $1/C$ is the eigenvalue.

Since $g$ is symmetric decreasing, the only possibility for $f$ is that 
it, too, must be decreasing (because $f$ equals $g$ plus the symmetric decreasing 
function
$\alpha \left(  C 2n -  |x|^2 
\right)$). 

It remains to clarify the consistency condition between the number $\rho 
$ and the number $C$, which appears twice. Thus, with $\mu(\rho)$ defined 
to be the fundamental eigenvalue of $-\Delta_R$ with the $\rho$ 
Robin boundary condition, the consistency condition is
\begin{equation}\label{consist}
 \mu \left( \frac {2}{1-2nC}\right) = \frac 1C \ .
\end{equation}
This equation picks out a unique $C$, and determines $f$, and $g$, up to a 
multiple.  The multiple is fixed by \eqref{elC}. 

As an example, consider $n=3 $.  Then $f(x) = |x|^{-1} \sin ( \lambda |x| )
$  and $1/C = \lambda^2$. (In dimensions other than 3, this function will
be replaced by a spherical 
Bessel function.) The $\rho $ Robin boundary condition 
is satisfied with $\rho = \lambda \cot \lambda -1$. The consistency 
condition \eqref{consist} becomes  
\begin{equation}\label{consist}
 \lambda^2 = \frac1C \quad \text{and} \quad 
 \lambda \cot \lambda -1 = 
 \frac{1}{1-6C}   \ . 
\end{equation}
Therefore $\lambda \cot \lambda = (6 - 2\lambda^2)/(6 - \lambda^2)$, and this has a single solution $\lambda_0 = 1.60412258...$
in the interval $(0,\sqrt{6})$.   Then if we define $\rho_0 = 
2/(1-6\lambda_0^2)$, and let $f$ be the fundamental $\rho_0$-Robin 
eigenfunction of $-\Delta$ on the unit ball $B$, and then define $g$ on $B$ 
by  $g = f + \alpha(|x|^2 -6\lambda_0^2)$, where $\alpha$ is chosen so that 
$g$ vanishes at the boundary of $B$, and define $ 0$ on the complement of 
$B$, $g$ is the an optimizer for the weighted Nash inequality (\ref{wnash}). 
  Having found the optimizer, the determination of the constant is reduced 
to quadrature.   The structure of the optimizers for (\ref{wnash}) is 
analogous to the of the optimizers for the original Nash inequality found in 
\cite{CL}:  They are sums of a multiple of the weight (constant for the 
original Nash inequality) and an eigenfunction of $-\Delta$ in a ball, 
except now with Robin boundary conditions in place of Neumann boundary 
conditions. 

To handle the $n-$dimensional case, define $J(r)$ to be the solution to
$J''(r) + \frac{n-1}{r}J'(r) =  - J(r)$ with $J(0)=1 $ and $J'(0) =0$. 
(This function can be expressed in terms of Bessel functions.)
The $n-$dimensional version of \eqref{consist}, after the elimination of 
$C$,  is
\begin{equation}\label{consist2}
 \frac{\lambda J'(\lambda ) }{ J(\lambda)} =
 \frac{\lambda^2}{\lambda^2-2n}   \ . 
\end{equation}
Let $\lambda_0 $ denote the least non-zero solution of \eqref{consist2} and
then define 
\begin{equation}\label{consist3}
 \rho_0= \frac{\lambda_0^2}{\lambda_0^2-2n} \  .
\end{equation}

We thus arrive at the following result: 

\begin{thm}[\bf Optimizers for the weighted Nash inequality]   \label{thm3}
Let $f$ be the principal (radial) eigenfunction of $-\Delta$ in the unit 
ball $B \subset \R^n$
with the $\rho_0-$Robin boundary 
condition, and
with $\rho_0$ defined in \eqref{consist3}. Normalize $f$ so that $f(|x|=1) 
= 2/\rho_0$.
Then the function 
$$
g(x) = \begin{cases}f(x)-(|x|^2 -1+ 2/\rho_0 )&  x \in B \\
                    \qquad 0 &  x \notin B
       \end{cases}
$$
 is an optimizer for the 
weighted Nash inequality \eqref{wnash}. Every optimizer has the form
$c g(\lambda  x)$ with $\lambda  >0$. 
\end{thm}

A similar argument holds for $|x|^{2k}$, for any $k$. One adds the function
$\Delta\left(\sum_{j=0}^k, \alpha_j |x|^{2j} \right)$, with suitable 
constants $\alpha_j$, to both sides of the analog of \eqref{elD}.


\begin{thebibliography}{30}


\bibitem{Au} {\sc T.~Aubin}, \textit{Problemes isop\'erim\'etriques et espaces de Sobolev.} Journal of differential geometry {\bf11}, No. 4 (1976), pp. 573-598.


\bibitem{CL} 
{\sc E.~A.~Carlen and M.~Loss} \textit{ Sharp constant in Nash's inequality}, Internat. Math. Res. Notices, {\bf 1993}, (1993) , pp.~213-215


\bibitem{DD} 
{\sc M.~Del Pino and J.~Dolbeault}, \textit{ Best constants
for Gagliardo-Nirenberg inequalities and applications to nonlinear
diffusions}, J. Math. Pures Appl., 81 (2002), pp.~847--875.


\bibitem{CFL} 
{\sc E.~A.~Carlen,R.~L.~Frank  and E.~H.~Lieb}, \textit{Stability estimates for the lowest eigenvalue of a Schr\"odinger operator},
G.A.F.A. {\bf 24}, (2014) pp.~63-84.


\bibitem{FL} 
{\sc R.~L.~Frank and E.~H.~Lieb}, \textit{A `liquid-solid' phase transition in a simple model for swarming, based on the `no flat-spots' theorem for subharmonic functions}, arXiv preprint arXiv:1607.07971.



\bibitem{G}
{\sc E.~Gagliardo}, 
\textit{ Propriet\`a di alcune classi di funzioni in pi\`u variabili} (Italian),
Ricerche Mat. {\bf 7} 1958 102--137. 

\bibitem{Lax}
{\sc  P.~D.~Lax}, \textit{ On Cauchys problem for hyperbolic equations and 
the differentiability of solutions of elliptic equations. }   Comm. Pure 
Appl. Math. {\bf 8} 1955, 615--633.

\bibitem{L83}
{\sc E.~H.~Lieb}, \textit{ Sharp constants in the
{H}ardy-{L}ittlewood-{S}obolev and related inequalities}, Ann.
Math. (2), 118 (1983), pp.~349--374.

\bibitem{LL} 
{\sc E.~H.~Lieb and M.~Loss} \textit{ Analysis, Second Ed}, 
Graduate Studies in Mathematics 14, A.M.S., Providence R.I., 2011

\bibitem{Na} 
{\sc J.~F.~Nash} \textit{Continuity of solutions of parabolic and elliptic equations}, Amer. J. Math., {\bf 80}, (1958),  PP.~931-954.


\bibitem{Ni}
{\sc L.~Nirenberg}, \textit{ On elliptic partial differential equations},  Ann. Scuola Norm. Sup. Pisa (3) {\bf 13}
(1959) pp.~115-162. 



\bibitem{R}
{\sc  Rockafellar},  \textit{Convex Analysis}, 
Princeton Univ. Press,  Princeton, 1970


\bibitem{Ta} {\sc G.~Talenti}, \textit{Best constants in Sobolev Inequality.} Ann. Mat. Pura Appl. \textbf{110} (1976), pp.~353-372.


\end{thebibliography}
\end{document}